\documentclass{amsart}
\usepackage{amsmath,amscd,amssymb,amsthm}

\RequirePackage{color}[1999/02/16 v1.0i Standard LaTeX Color (DPC)]

\newif\ifcomments

\let\newComments\newKibitzer

\newComments\Ov{Ov}{red}
\newComments\Bj{Bj}{blue}
\newComments\DL{DL}{red}

\newtheorem{lemma}{Lemma}[subsection]

\newtheorem{defi}[lemma]{Definition}
\newtheorem{ex}[lemma]{Example}
\newtheorem{rem}[lemma]{Remark}
\newtheorem{prop}[lemma]{Proposition}
\newtheorem{thm}[lemma]{Theorem}
\newtheorem{cor}[lemma]{Corollary}

  \font\bib=cmbx8 \font\eightit=cmti8

\def\oh{{\ts\frac{1}{2}}}

\def\cal{\mathcal}

\def\D{{\cal D}}

\def\F{{\cal F}}

\def\Vect{\mathrm{Vect}}

\def\Bbb{\mathbb}

\def\bN{\Bbb N}

\def\bR{\Bbb R}

\def\l{\lambda}

\def\Div{\mathop{\rm Div}\nolimits}

\def\oh{{\ts\frac{1}{2}}}

\def\ts{\textstyle}

\def\vol{\omega}

\title[Riemannian Curl in Contact Geometry]{Riemannian Curl in Contact Geometry}

\author{Sofiane Bouarroudj}
\address{Division of Science and Mathematics
\\New York University Abu Dhabi
\\Po Box 129188, United Arab Emirates}
\email{sofiane.bouarroudj@nyu.edu}

\author{Valentin Ovsienko}
\address{CNRS, Institut Camille Jordan
\\Universit\'e Claude Bernard Lyon~I
\\21 Avenue Claude Bernard, 69622 Villeurbanne Cedex, France}
\email{ovsienko@math.univ-lyon1.fr}

\begin{document}

\begin{abstract}
We consider a contact manifold with a pseudo-Riemannian metric and
define a contact vector field intrinsically associated to this pair of structures. 
We call this new differential invariant  the contact Riemannian curl. 
On a Riemannian manifold, Killing vector fields are those that annihilate the metric;
a Killing $1$-form is obtained from a Killing vector field by lowering indices.
We show that the contact Riemannian curl vanishes
if the metric is of constant curvature and the contact structure is defined by a Killing $1$-form.
We also show that the contact Riemannian curl has
a strong similarity with the Schwarzian derivative since it depends
only on the projective equivalence class of the metric. 
For the Laplace-Beltrami operator on a contact manifold,
the contact Riemannian curl is proportional to the subsymbol defined
in~arXiv:1205.6562. 
We also show that the contact Riemannian curl vanishes
on the (co)tangent bundle over a Riemannian manifold.
This implies that the corresponding subsymbol of the Laplace-Beltrami operator is identically zero.
\end{abstract}


\maketitle

\section{Introduction}  \label{Intro}

The principal object of this paper is related to
the notion of {\it invariant differential operator}, i.e., an
operator commuting with the action of the group of diffeomorphisms.
The notion of {\it differential invariant} is one of the oldest notions
of differential geometry.
The best known example is perhaps the curvature in all its avatars.
The topic to which the present work belongs was initiated by Veblen~\cite{Veb} who started a
systematic study of invariant differential operators on smooth manifolds.
The theory was intensively studied in the 80's in the context of
Gelfand-Fuchs cohomology; see~\cite{Fu86,GLS} and references
therein.

We consider a smooth manifold $M$ equipped simultaneously with
a contact structure and a pseudo-Riemannian metric.
We present a construction of a contact vector field
corresponding to these two structures;
we call this vector field the {\it contact Riemannian curl}.
Our construction is coordinate free and invariant
with respect to the action of the group of contact diffeomorphisms,
i.e., the contact Riemannian curl is a differential invariant.
Moreover, our goal is to define this differential invariant in a ``most symmetric'' way,
so that it is also invariant
with respect to natural equivalence relations.

One of the equivalence relations we consider is as follows.
Two metrics are called {\it projectively equivalent}
(or geodesically equivalent) if they have the same non-parametrized geodesics.
i.e., their Levi-Civita connections are projectively equivalent.
It turns out that the constructed contact Riemannian curl
is obtained as contraction of the metric with
a certain tensor field invariant with respect to this equivalence relation.
This implies, in particular, that the contact Riemannian curl of the pair 
(a metric of constant scalar curvature, contact structure defined by a Killing $1$-form) 
vanishes.
Projective invariance makes the notion of contact Riemannian curl
quite similar to that of classical {\it Schwarzian derivative}
(for various multi-dimensional generalizations
of the Schwarzian derivative see~\cite{BO,OT,B,OT1} and references therein).
We investigate this relation in more details.

Among the main properties of the contact Riemannian curl that we investigate,
there is its relation to the Laplace-Beltrami operator.
Differential operators on contact manifolds
were studied from the geometric point of view in a recent work \cite{CO12},
where the notion of {\it subsymbol} of a differential operator on a contact manifold
was introduced.
The subsymbol of a differential operator is a tensor field of degree lower
than that of the principal symbol.
Note that the subsymbol is not well-defined for an arbitrary manifold,
one needs a contact structure to obtain an invariant definition.
For a given second order differential operator, the subsymbol is just a contact vector field.
In the present paper, we consider the {\it Laplace-Beltrami operator}
associated to an arbitrary metric on a contact manifold
and calculate its subsymbol.
It turns out that this subsymbol is proportional to the contact Riemannian curl.

We also apply our general construction to a particularly
interesting example of a manifold that has natural
contact and Riemannian structures,
namely to the spherical
(or projectivized) cotangent bundle $ST^*M$
over a Riemannian manifold $(M,g)$.
The manifold $ST^*M$ is equipped with the canonical
 lift of the metric $g$.
We show that the contact Riemannian curl, 
and therefore the subsymbol of the Laplace-Beltrami operator,
is identically zero in this case.
Let us mention that
the projectivization of the cotangent bundle over a Riemannian manifold $M$,
as well as the sphere bundle $ST^*M$,
is an example of a ``real-complex'' manifold whose local invariants
were recently introduced and computed in~\cite{BGLS}.

At the end of the paper, we provide
several concrete examples of the contact Riemannian curls.
For instance, we calculate it for
the $3{\rm D}$-ellipsoid equipped with the conformally flat metric
introduced in \cite{Tab} and intensively used in~\cite{MT,DV}.

We believe that the differential invariants of a pair 
(a Riemannian metric, a contact structure) is worth a
systematic study.

\section{Contact geometry and tensor fields}  \label{CODG}

Contact geometry is an old classical subject,
that can be viewed as an odd-dimensional version of symplectic geometry.
Let $M$ be a contact manifold and $\dim(M)= 2\ell + 1$,
we will always assume that $\ell\geq1$.
Unlike a symplectic structure in symplectic geometry, a contact structure on $M$
is defined by a differential $1$-form $\theta$, called a {\it contact form},
determined up to a factor (a function),
and such that $d\theta$ is a $2$-form of rank $2\ell$.
It is important that a contact form is not intrinsically associated with
a contact structure.

A contact diffeomorphism (a contact vector field) is a diffeomorphism
(a vector field) preserving the contact structure.
It preserves a given contact form conformally, up to a factor.
The space of all contact vector fields can be identified
with the space of smooth functions, but
this correspondence depends on the choice of a contact form; see~\cite{Arn}.

In this section, we recall several standard facts of contact geometry ---
those of contact structure and contact vector fields ---
using somewhat unconventional notation of~\cite{OT,Ovs}
which are among our main references.
We show that the contact structure can be also described by a special tensor field,
which is a {\it weighted contact form}.
Contact vector fields are in one-to-one correspondence with
weighted densities of weight $-\frac{1}{\ell+1}$.

\subsection{Weighted densities}

A weighted density is a standard object in differential geometry.
In~order to make the definitions intrinsic, we recall here this notion.

Let $M$ be a manifold of dimension $n$.
For any $\l\in\bR$, we denote by $(\Lambda^n T^*M)^{\otimes\l}$
the line bundle of homogeneous complex valued functions of weight $\l$
on the determinant bundle $\Lambda^n TM$.
The space $\F_\l(M)$ of smooth sections of
$(\Lambda^n T^*M)^{\otimes\l}$ with complex coefficients
is called the space of {\it weighted densities} of weight~$\lambda$,
(or $\l$-{\it densities} for short).

\begin{ex}
{\rm
If the manifold $M$ is orientable and
if $\vol$ is a volume form with constant coefficients, then $\phi\,\vol^\l$,
where $\phi\in{}C^\infty(M)$, is a $\l$-density.
}
\end{ex}

The space $\F_\l(M)$ has the structure of a module over the Lie algebra
$\Vect(M)$ of all smooth vector fields on $M$.
We denote by $\mathrm{Div}$ the divergence operator
associated with a volume form $\vol$ on $M$.
That is, $L_X(\vol)=\mathrm{Div}(X)\,\vol$.
The action of a vector fields reads as follows:
\begin{equation}
\label{LieEq}
L_X(\phi\,\vol^\l)=
\left(X(\phi)+\l\mathrm{Div}(X) \phi
\right)\vol^\l,
\end{equation}
for every vector field $X$ and $\phi \in C^\infty(M)$.

\subsection{Contact manifolds}

A smooth manifold $M$ is called {\it contact} if it is equipped with
a completely non-integrable distribution
$$
\mathcal{D}\subset{}TM
$$
of codimension~$1$. The distribution $\mathcal{D}$ is called  a {\it
contact distribution}; the hyperplane $\mathcal{D}_x\subset{}T_xM$
is called a {\it contact hyperplane} for every point $x\in{}M$.
A contact structure on $M$ exists only if $\dim M=2\ell+1>1$.

 A usual way to define a contact structure is to chose
 a (locally defined) differential $1$-form $\theta$ on~$M$ such that
 $\mathcal{D}=\ker\theta$.
Such a $1$-form is called a {\it contact form}. The complete
non-integrability  of the distribution $\mathcal{D}$ is equivalent
to the fact that
\begin{equation}
\label{TheVol}
\mathrm{vol}:=\theta\wedge(d\theta)^\ell
\end{equation}
is a (locally defined) volume form;
equivalently, the 2-form $d\theta$ is a non-degenerate on the contact hyperplanes
$\mathcal{D}_x$ of $\mathcal{D}$.
However, there is no canonical choice of a contact form.

A diffeomorphism $f:M\to{}M$ is a {\it contact diffeomorphism} if $f$
preserves $\mathcal{D}$.
If $\theta$ is a contact form corresponding to the contact distribution $\mathcal{D}$
and $f$ is a contact diffeomorphism, then $f$ does not necessarily preserve $\theta$,
more precisely, $f^*\theta=F_f\theta$, where $F_f$ is a function.

We refer to \cite{Arn,Bla} for excellent textbooks on contact geometry.

\subsection{The contact tensor}

We will be using the notion of a (generalized) tensor field 
that was suggested in \cite{BL} and goes back to ideas of I. M. Gelfand.
Besides the standard tensor fields, i.e., sections of the bundles\footnote{
Throughout this paper, the tensor product is performed over $C^\infty(M)$.}
$(TM)^p\otimes(T^*M)^q$, it is often useful to consider
{\it weighted} tensor fields that are sections of the bundles
$$
(TM)^p\otimes(T^*M)^q\otimes(\Lambda^n T^*M)^{\otimes\l}.
$$
A wealth of examples of such generalized tensor fields
can be found in \cite{Fu86,OT}.

We are ready to introduce the main notion of this section.

\begin{defi}
\label{ConTD}
{\rm
Given a contact form $\theta$, let
the {\it contact tensor field} be
\begin{equation}
\label{CT}
\Theta:=\theta\otimes\mathrm{vol}^{-\frac{1}{\ell+1}},
\end{equation}
where $\mathrm{vol}$ is as in Eq. (\ref{TheVol}).
}
\end{defi}

\begin{prop}
\label{IPro}
The tensor field $\Theta$ is globally defined on a contact manifold $M$,
it is independent of the choice of a contact form,
and it is invariant with respect to the contact diffeomorphisms.
\end{prop}

\begin{proof}
Let $F$ be a non-vanishing function and consider the contact form $F\theta$.
The corresponding volume form is
$F\theta\wedge\left(d(F\theta)\right)^\ell=F^{\ell+1}\theta\wedge(d\theta)^\ell$.
Therefore, the contact tensor fields defined by Eq.~(\ref{CT}),
 corresponding to the contact forms
$\theta$ and $F\theta$, coincide.
Hence, $\Theta$ is globally defined and invariant with respect to
contact diffeomorphisms.
\end{proof}

A contact structure is intrinsically defined by the corresponding contact tensor.

\begin{ex}
\label{Dabex}
{\rm
Local coordinates $(x^1,\ldots,x^\ell,y^1,\ldots,y^\ell,z)$
on $M$ are often called the {\it Darboux coordinates} if the
contact structure can be represented by the 1-form
\[
\theta_{\mathrm{Dar}}=dz+\frac{1}{2}\,\sum_{i=1}^\ell \left (x^idy^i-y^idx^i \right ).
\]
The corresponding volume form is then the standard one:
$$
\mathrm{vol}=
(-1)^{\frac{\ell(\ell-1)}{2}}\,\ell!\,
dx^1\wedge\cdots\wedge{}dx^\ell\wedge{}dy^1\wedge\cdots{}dy^\ell\wedge{}dz.
$$
A contact structure has no local invariants, therefore
Darboux coordinates always exist in the vicinity of every point;
see~\cite{Arn} (and~\cite{GL} for a simple algebraic proof).
}
\end{ex}

\subsection{Contact vector fields}\label{CoHSect}

A {\it contact vector field} on a contact manifold $M$ is a vector field
that preserves the contact distribution.
This is usually expressed in terms of contact forms:
a vector field $X$ is contact if, for every contact form $\theta$, the Lie derivative
$L_X\theta$ is proportional to $\theta$:
\begin{equation}
\label{DivEq}
   L_X\theta = {\ts \frac{1}{\ell+1}}\Div(X) \theta.
\end{equation}
In terms of the contact tensor (\ref{CT}), we have the following
corollary of Proposition~\ref{IPro}.
\begin{cor}
\label{ContV}
A vector field $X$ is contact if and only if it preserves the contact tensor:
$$
L_X\Theta=0.
$$
\end{cor}

Let ${\cal K}(M)$ denote the space of all smooth contact vector fields on $M$.
This space has a Lie algebra structure, it is also a module over the group
of contact diffeomorphisms.
The following observation can be found in~\cite{Ovs,CO12}.

\begin{prop}
\label{VecPro}
As a module over the group  of contact diffeomorphisms,
the space ${\cal K}(M)$ is isomorphic to the space of weighted densities
${\cal F}_{-\frac{1}{\ell+1}}(M)$.
\end{prop}

\begin{proof}
The space of contact forms is isomorphic to ${\cal F}_{\frac{1}{\ell+1}}(M)$.
Indeed, this follows from Proposition~\ref{IPro} and from Eq.~(\ref{DivEq}).
The statement then follows from the fact that
there is a natural $C^\infty(M)$-valued pairing between the spaces of contact vector fields
and of contact forms:
$
 (X,\,\theta)\mapsto\theta(X).
$
\end{proof}

\begin{rem}
{\rm
The above proposition means that, unlike the symplectic geometry,
 the notion of contact generating function (or ``contact Hamiltonian function'')
should be understood as a weighted density
and not as a function.
However, in the Darboux coordinates, the correspondence
between the elements of ${\cal K}(M)$ and ${\cal F}_{-\frac{1}{\ell+1}}(M)$
becomes the usual correspondence between
contact vector fields and functions (see~\cite{Arn}):
$$
X_{\phi\,\omega^{-\frac{1}{\ell+1}}}=
\sum_{i=1}^\ell\left(
\partial_{x^i}(\phi)\,\partial_{y^i}-\partial_{y^i}(\phi)\,\partial_{x^i}
\right)
\textstyle
+\frac{1}{2}\,\partial_{z}(\phi)\,{\cal E}
+\left(\phi-\frac{1}{2}\,{\cal E}(\phi)\right)\partial_{z},
$$
where
$$
{\cal E}=
\sum_{i=1}^\ell
\left(
x^i\partial_{x^i}+y^i\partial_{y^i}
\right)
$$
is the Euler vector field.
}
\end{rem}

\begin{ex}
{\rm
If $\dim{M}=3$, contact vector fields are
identified with $-\frac{1}{2}$-densities;
if $\dim{M}=5$, then ${\cal K}(M)\cong{\cal F}_{-\frac{1}{3}}(M)$, etc.
Note also that, in the one-dimensional case,
every vector field is contact,
one then has $\Vect(M)\cong{\cal F}_{-1}(M)$.
}
\end{ex}

\subsection{Another definition of weighted densities on contact manifolds}\label{Mitia}

In presence of a contact structure defined by a contact form
$\theta$, it is natural to express elements of any rank 1 bundle,
for example, weighted densities, in terms of powers of $\theta$:
$$
\phi\,\mathrm{vol}^{\frac{\l}{\ell+1}}\longleftrightarrow\phi\theta^\l,
$$
where as above $\phi$ is a smooth function.
The notation $\phi\theta^\l$ is adopted in many works by physicists
(see also~\cite{Ovs90,GLS01}).
In this notation, many formulas simplify.
For instance, if $X$ is a contact vector field,
then the corresponding contact Hamiltonian is
$\phi\theta^{-1}$, where the function $\phi$
is simply the evaluation $\phi=\theta(X)$.

\subsection{The Poisson algebra of weighted densities}

The space $\F(M)=\bigoplus_\l\F_\l(M)$
of all weighted densities on a contact manifold $M$ can be endowed with a structure of a Poisson algebra
(see~\cite{Arn,OT,Ovs}):
$$
\{.,.\}:{\cal F}_\l(M)\times{\cal F}_{\mu}(M)\to{\cal F}_{\l+\mu+\frac{1}{\ell+1}}(M).
$$
The explicit formula in Darboux coordinates is as follows:
$$
\left\{\phi\,\omega^\l,\psi\,\omega^{\mu}\right\}=
\left(
\sum_{i=1}^n(\partial_{x^i}\phi\,\partial_{y^i}\psi -
\partial_{x^i}\psi\,\partial_{y^i}\phi)
+\partial_{z}\phi\left(\mu\psi+{\cal E}\psi\right)
-\partial_{z}\psi\left(\l\phi+{\cal E}\psi\right)
\right)\omega^{\l+\mu+\frac{1}{\ell+1}}.
$$

The subspace ${\cal F}_{-\frac{1}{\ell+1}}(M)$ is a Lie subalgebra of ${\cal F}$
isomorphic to ${\cal K}(M)$.
The Poisson bracket of $-\frac{1}{\ell+1}$-densities precisely corresponds
to the Lie derivative:
$$
X_{\left\{\Phi,\Psi\right\}}=
L_{X_{\Phi}}
\left(\Psi\right),
$$
where $\Phi=\phi\,\omega^{-\frac{1}{\ell+1}},\,\Psi=\psi\,\omega^{-\frac{1}{\ell+1}}$.

\subsection{The invariant splitting}

The full space of vector fields $\mathrm{Vect}(M)$ splits into direct sum
$$
\mathrm{Vect}(M)=
{\cal K}(M)\oplus{\cal Tan}(M),
$$
where ${\cal Tan}(M)$ is the space of vector fields tangent
to the contact distribution,
i.e., $\theta(Y)=0$ for every contact form $\theta$ and every $Y\in{\cal Tan}(M)$.
Such vector fields are called {\it tangent vector fields}.
Unlike ${\cal K}(M)$, the space  ${\cal Tan}(M)$ is not a Lie
algebra, but a ${\cal K}(M)$-module.

The above splitting is invariant with respect to the group of contact diffeomorphisms.
In particular, there is an invariant projection
\begin{equation}
\label{PiPr}
\pi:\mathrm{Vect}(M)\to{\cal K}(M),
\end{equation}
that will be very useful.

\section{The contact Riemannian curl and its properties}  \label{Definitions}

In this section, we introduce our main notion, a contact vector field
corresponding to a metric and a contact structure.
We also study its main properties, such as projective invariance
and relation to the multi-dimensional Schwarzian derivative.

\subsection{Covariant derivative}
Let us assume now that $M$ is endowed with a pseudo-Riemannian
metric $g$. We denote the Levi-Civita connection on $M$ by
$\nabla$, and  the Christoffel symbols by $\Gamma_{ij}^k$. The {\it
covariant derivative}, also denoted by $\nabla$, is the linear map
that can be defined for arbitrary space of tensor fields, ${\cal
T}(M)$: 
$$
\nabla:{\cal T}(M)\to\Omega^1(M)\otimes{\cal T}(M),
$$
such that 
$\nabla(fm)=df\otimes m+f\otimes\nabla(m)$ for any $f\in{}C^\infty(M)$ and $m\in T(M)$.
It is written in the form $\nabla(t)=\nabla_i(t)\,dx^i$,
and therefore it suffices to define the partial derivatives~$\nabla_i$.
Here and below summation over repeated indices
(one upper, the other one lower) is understood (Einstein's notation);
see~\cite{DNF}.

The covariant derivative
of vector fields and differential $1$-forms
is given, in local coordinates, by the well-known formulas
$$
\nabla_i \left(V^j\partial_j\right)=
\left(\partial_iV^j+\Gamma_{ik}^jV^k\right)\partial_j,
\qquad
\nabla_i \left(\beta_jdx^j\right)=
\left(\partial_i\beta_j-\Gamma_{ij}^k\beta_k\right)dx^j,
$$
respectively,
where $\partial_i=\partial/\partial{}x^i$.
The covariant derivative then extended to every tensor fields by Leibniz rule.

For instance, the covariant derivative of weighted densities is
defined in local coordinates by the following formula:
\[
\nabla_i \left(
\phi\,\omega^\l
\right)=
\left(\partial_i \phi-\l \Gamma_{ij}^j\phi\right)\omega^\l,
\]
that we will extensively use throughout the paper.

\subsection{The main definition}
Let us introduce the main notion of this paper.
Recall that the contact tensor field $\Theta$ was introduced in
Definition~\ref{ConTD}.

\begin{defi}
\label{MainDef}
{\rm
(a)
For every pseudo-Riemannian metric $g$ on a contact manifold $M$,
we define a weighted density of degree $-\frac{1}{\ell+1}$:
\begin{equation}
\label{maindefi}
A_{g,\Theta}:=
\left\langle
{g},\,\nabla\Theta
\right\rangle,
\end{equation}
in local coordinates, $A_{g,\Theta}:=
{g}^{ij}\nabla_i\Theta_j$.

(b)
We call the contact vector field $X_{A_{g,\Theta}}$
with contact Hamiltonian $A_{g,\Theta}$
the {\it contact Riemannian curl of $g$}.
}
\end{defi}

Note that the quantity $A_{g,\Theta}$ is, indeed, a weighted density
of degree $-\frac{1}{\ell+1}$, so that it has a meaning
of contact Hamiltonian; see Proposition~\ref{VecPro}.

\begin{rem}
{\rm
The tensor field $\nabla\Theta$ is also a differential invariant
(that actually contains even more information than $A_{g,\Theta}$).
One can obtain a $-\frac{1}{\ell+1}$-density out of $\nabla\Theta$
by contracting with an arbitrary metric, not necessarily with $g$ itself.
}
\end{rem}

It will be useful to have an explicit expression
for $A_{g,\Theta}$ (and of $\nabla\Theta$) in local coordinates.

\begin{prop}
\label{LoP}
In local coordinates, such that $\Theta=\theta\otimes\mathrm{vol}^{-\frac{1}{\ell+1}}$,
one has
\begin{equation}
\label{ProCurl}
 A_{g,\Theta}=
 g^{ij}
  \left (
  \partial_i \theta_j-
 \left (\Gamma^k_{ij}-\frac{1}{2(\ell+1)}
 \left(\delta^k_i\Gamma_{jr}^r+\delta^k_j\Gamma_{ir}^r\right )\right )\theta_k\right)
 \mathrm{vol}^{-\frac{1}{\ell+1}},
\end{equation}
where $\delta^k_i$ is the Kronecker symbol.
\end{prop}

\begin{proof}
This can be obtained directly from Definition \ref{MainDef}
and the expression of the covariant derivative of a weighted density.
\end{proof}

\begin{rem}
{\rm
It follows from the intrinsic definition (\ref{maindefi}) that
the local expression (\ref{ProCurl}) is actually invariant
with respect to the action of the group of contact diffeomorphisms.
The formula (\ref{ProCurl}) remains unchanged for any choice
of local coordinates.
It is also independent of the choice of the contact form.
}
\end{rem}

\subsection{Projective invariance of $\nabla\Theta$}

Let us recall a fundamental notion of projectively equivalent connections
due to Cartan~\cite{Car}.

A {\it projective connection} is an equivalence class of symmetric
affine connections giving the same non-parameterized geodesics.
The {\it symbol of a projective connection} is
given by the expression
\[
\textstyle
\Pi_{ij}^k:=\Gamma_{ij}^k-\frac{1}{n+1}\left
(\delta_i^k \Gamma_{lj}^l+\delta_j^k \Gamma_{il}^l\right ),
\]
where $n$ is the dimension; see~\cite{KN}.
Note that in our case, $n=2\ell+1$.

The simplest properties of a projective connection are as the following.

\begin{enumerate}
\item
Two affine connections, $\nabla$ and $\tilde \nabla$, are projectively
equivalent if and only if $\Pi_{ij}^k=\tilde\Pi_{ij}^k$.
\item
Equivalently, $\nabla$ and $\tilde \nabla$ are projectively
equivalent if and only if there exists a 1-form $\beta$ such that
\begin{equation*}
\label{assoc}
\tilde
\Gamma_{ij}^k=\Gamma_{ij}^k+\delta_{j}^k\,\beta_i+\delta_{i}^k\,\beta_j.
\end{equation*}
\end{enumerate}

The following statement makes the contact Riemannian curl
somewhat similar to the Schwarzian derivative.

\begin{thm}
\label{CoCuProj}
If $g$ and $\tilde g$ are two metrics
whose Levi-Civita connections are projectively equivalent,
then $\nabla\Theta=\tilde \nabla\Theta$.
\end{thm}

\begin{proof}
The coordinate formula for $\nabla\Theta$
can be written as follows: 
$$
\left(\nabla\Theta\right)_{ij}=
  \left (
  \partial_i \theta_j-
\Pi^k_{ij}\,\theta_k\right) \mathrm{vol}^{-\frac{1}{\ell+1}},
$$
see  (\ref{ProCurl}).
This expression depends only on the projective class of the Levi-Civita connection
and implies projective invariance.
\end{proof}

Let $[g]$ denote the class of geodesically equivalent metrics,
let $[\nabla]$ denote the corresponding projective connection.
The above theorem means that the tensor $\nabla\Theta$ actually depends
only on $[g]$ and not on the metric itself.

\begin{rem}
{\rm
Geodesically equivalent metrics is a very classical subject of 
Riemannian geometry that goes back to Beltrami, Levi-Civita, Weyl, and Cartan.
We refer to the classical book~\cite{Eis} for a survey. 
The subject is still very active, see~\cite{BKM} and references therein.
}
\end{rem}

\subsection{Projectively flat connections,
metrics of constant curvature and Killing contact forms}
It is now natural to investigate projectively flat case.

A  connection $\nabla$ on $M$ is called \textit{projectively flat} if, in a
neighborhood of every point, there exist local coordinates,
often called {\it adapted coordinates}, such
that the geodesics are straight lines in these coordinates.
If a connection is projectively flat, then $\Pi_{ij}^k\equiv0$
in any system of adapted coordinates.

Note also that projectively flat connections
 admit a (local) action of the group $\mathrm{SL}(n+1,\bR)$,
 in other words, adapted coordinates admit linear-fractional
 changes.

The classical Beltrami theorem states that {\sl the Levi-Civita connection
of a Riemannian metric is projectively flat
if and only if the metric has a constant sectional curvature}.
This fact allows us to obtain an important consequence of Theorem~\ref{CoCuProj}.

Let us recall the notion of Killing differential forms that goes back to Yano~\cite{Yan}.
A $1$-form $\beta=\beta_i(x)dx^i$ is said to be a {\it Killing form} if
$$
\nabla_i \beta_j+\nabla_j \beta_i=0.
$$
Recall also a more common notion of Killing
vector field.
A vector field $V=V^i(x)\partial_i$ is said to be a {\it Killing vector field} if
$$
L_Vg=0
$$
Every Killing $1$-form can be obtained from a Killing
vector field by lowering indices: $\beta=\langle g,V\rangle$;
i.e., $\beta_i=g_{ij}V^j$ in local coordinates.

\begin{cor}
\label{CoCu}
If $g$ is a metric of constant sectional curvature
and if the contact structure is defined by a contact $1$-form $\theta$
which is a Killing form with respect to a metric from the projective class $[g]$,
then $ A_{g,\Theta}=0$.
\end{cor}

\begin{proof}
Since the Levi-Civita connection
corresponding to $g$ is projectively flat,
there exist local coordinates for which
$\Pi_{ij}^k\equiv0$, and therefore
$$
A_{g,\Theta}=g^{ij}\partial_i \theta_j.
$$
If, furthermore,
$
\partial_i \theta_j+\partial_j \theta_i=0
$
for all $i,j$, then $A_{g,\Theta}$ vanishes identically
since the tensor $g^{ij}$ is symmetric.
The equation $\partial_i \theta_j+\partial_j \theta_i=0$
means that~$\theta$ is a Killing form with respect to
the flat metric which is projectively equivalent to $g$.

The corollary then follows from Theorem \ref{CoCuProj}.
\end{proof}

\begin{ex}
{\rm
The Darboux form in Example \ref{Dabex} is a Killing form with respect
to the flat metric.
Note that in other works, especially in those on analytical mechanics, another
local normal form of the contact form is often used:
$dz+\sum_{1\leq{}i\leq{}\ell}x^idy^i$.
(Over fields of characteristic $2$, only this latter form can be used, see \cite{Leb}.)
However, this is not a Killing form with respect to the flat metric.
}
\end{ex}

\subsection{Contact equivariance}
Consider the action of the group of contact diffeomorphisms.
It immediately follows from the intrinsic (i.e., invariant) definition
(\ref{maindefi}) of $A_{g,\Theta}$ of  that the map
$g\mapsto{}A_{g,\Theta}$
from the space of metrics to that of $-\frac{1}{\ell+1}$-densities
commutes with this action:
\begin{equation}
\label{CoCAct}
A_{f^*g, \Theta}=
f^*\left(A_{g,\Theta}\right).
\end{equation}
From this fact and  Corollary~\ref{CoCu}, we deduce the following statement.

\begin{cor}
\label{SecCor}
If a metric $\tilde g $ is contactomorphic to a metric $g $
of constant sectional curvature
and if the contact structure is defined by a contact $1$-form $\theta$
which is a Killing form with respect to 
 $g$, then $A_{\tilde g, \Theta}=0$.
\end{cor}

\subsection{Action of the full group of diffeomorphisms}
Let us consider the action of the group of all diffeomorphisms.
It turns out that this action is related to a quite remarkable $1$-cocycle.

Recall that the space of connections is an affine space associated with the space of
$(2,1)$-tensor fields, i.e.,
given two connections, $\nabla$ and $\tilde\nabla$, the difference
$\nabla-\tilde\nabla$ is a well-defined
$(2,1)$-tensor field.
This allows one to define a $1$-cocycle on
the group of all diffeomorphisms.
If $f$ is an arbitrary, not necessarily contact, diffeomorphism, we set:
$$
C(f):=f^*\nabla-\nabla,
$$
where $\nabla$ is an arbitrary fixed connection, choice of which
changes $C$ by a coboundary\footnote{
Note also that the cocycle $C$ provides a universal way to construct
representatives of non-trivial classes of the Gelfand-Fuchs cohomology;
see~\cite{Gel}.}.

Let $\nabla$ and $\tilde\nabla$ be two connections on~$M.$
The difference of the projective equivalence classes $[\nabla]-[\tilde\nabla]$
can be understood as a traceless
$(2,1)$-tensor field.
Therefore, a projective
connection on $M$ leads to the following $1$-{\it cocycle} on
the group of all diffeomorphisms:
$$
{\mathfrak T}(f)=f^*[\nabla]-[\nabla]
$$
which vanishes on (locally) projective diffeomorphisms.
In local coordinates,
$$
{\mathfrak T}(f)^k_{ij}:= f^*\Pi_{ij}^k-\Pi_{ij}^k,
$$
where $\Pi_{ij}^k$ are the projective Christoffel symbols\footnote{
The $1$-cocycle ${\mathfrak T}$ is often considered as a higher-dimensional
analog of the Schwarzian derivative; see~\cite{OT}.
If~$\nabla$ is projectively flat, then
the group $\mathrm{SL}(n+1,\bR)$
of (local) symmetries of $[\nabla]$ is precisely the kernel of ${\mathfrak T}$.}.

\begin{prop}
\label{CoCActThm}
If $f:M\to{}M$ is an arbitrary diffeomorphism, then
\begin{equation}
\label{CoCActArb}
f^*\left(A_{g,\Theta}\right)- A_{f^*g, \Theta}=
 f^* \left\langle g, \nabla\Theta\right \rangle
 -\left\langle f^*g,
 \nabla\Theta\right \rangle
+\left\langle  f^*g\otimes \Theta,  {\mathfrak T}(f)\right\rangle.
\end{equation}
\end{prop}

\begin{proof}
Let us first clarify the notation.
Since ${\mathfrak T}(f)$ is a $(2,1)$-tensor field, the pairing
$\left\langle g\otimes\Theta,\,{\mathfrak T}(f)\right\rangle$
is well-defined.
Furthermore, taking into account the weight of the contact tensor $\Theta$,
it follows that
$\left\langle g\otimes\Theta,\,{\mathfrak T}(f)\right\rangle$
is a weighted density of weight~$-\frac{1}{\ell+1}$.

In local coordinates and using Proposition \ref{LoP}, we have
\[
\begin{array}{lcl}
A_{f^*g,\Theta}&=&(f^* g)^{ij}\left (\partial_i \theta_j-f^*\Pi^k_{ij}\, \theta_k \right )\\[2mm]
       &=& (f^* g)^{ij}\left (\partial_i \theta_j-(f^*\Pi^k_{ij}-\Pi^k_{ij})\, \theta_k \right )- (f^* g)^{ij}\Pi^k_{ij}\theta_k\\[2mm]
&= &  (f^* g)^{ij}\left (\partial_i \theta_j-{\mathfrak T}(f)^k_{ij}\, \theta_k \right )- (f^* g)^{ij}\Pi^k_{ij}\theta_k\\[2mm]
&=& (f^* g)^{ij}\left (\partial_i \theta_j-\Pi^k_{ij}\theta_k \right )-(f^* g)^{ij}\,{\mathfrak T}(f)^k_{ij}\, \theta_k\\[2mm]
&=& (f^* g)^{ij}\nabla_i (\Theta_j)-(f^* g)^{ij}\,{\mathfrak T}(f)^k_{ij}\, \theta_k\\[2mm]
&=&
\left\langle f^*g,
 \nabla \left ( \Theta\right )\right \rangle
 -\left\langle f^*g\otimes\Theta,\,  {\mathfrak T}(f)\right\rangle.
\end{array}
\]
It remains to notice that
$f^*\left(A_{g,\Theta}\right)=f^* \left\langle g,
\nabla \left ( \Theta \right )\right \rangle$.
Proposition~\ref{CoCActThm} is proved.
\end{proof}

\section{The subsymbol of the Laplace-Beltrami operator}

In this section, we explain the relation of the Riemannian curl to
the classical Laplace-Beltrami operator.
Let us mention that study of
differential operators on contact manifolds is a classical subject;
see a recent work~\cite{vE10} and references therein.

\subsection{Differential operators and diffeomorphism action}

Let $M$ be an arbitrary smooth mani\-fold and
$\D_{\l, \mu}(M)$ be the space of linear differential operators
acting on the space of weighted densities:
$$
T:\F_\l(M)\to\F_\mu(M).
$$
The space $\D_{\l, \mu}(M)$ is naturally a module over the group of diffeomorphisms,
the module structure being
dependent of the weights $\l$ and $\mu$.
For $k \in \bN$, let $\D^k_{\l, \mu}(M)$ be the space of
linear differential operators of order~$\le k$.
The spaces $\D^k_{\l, \mu}(M)$ define a filtration on $\D_{\l, \mu}(M)$
invariant with respect to the group of diffeomorphisms.

Recall the classical notion of {\it symbol} (or the {\it principal symbol}) of
a differential operator of order~$k$.
It is defined as the image of the projection
$$
\sigma:\D_{\l, \mu}(M)\to\D^k_{\l, \mu}(M)/\D^{k-1}_{\l, \mu}(M).
$$
Observe that, in the particular case $\l=\mu$,
the quotient space $\D^k_{\l, \l}(M)/\D^{k-1}_{\l, \l}(M)$
can be identified with the space of
symmetric contravariant tensor fields of degree $k$ on $M$.

We will be especially interested in the space $\D^2_{\l, \l}(M)$
of $2$-nd order operators acting on $\l$-densities;
a systematic study of this space viewed as
a module over the group of diffeomorphisms was initiated in~\cite{DO1}.

\subsection{The subsymbol of a second order differential operator}
In~\cite{CO12}, the space of differential operators
on a contact manifold was studied as a module over
the group of contact diffeomorphisms.
It was proved that there exists
a notion of {\it subsymbol} which is a tensor field
of degree lower than that of the principal symbol.

For a $2$-nd order differential operator,
the subsymbol is just a contact vector field.
More precisely,
for every $\l$, there exists a linear map
(which is unique up to a constant factor)
$$
{\mathrm s}\sigma:
{\cal D}^2_{\lambda,\lambda}(M) \rightarrow {\cal K}(M),
$$
invariant with respect to the action of the group of contact
diffeomorphisms. The image ${\mathrm s}\sigma(T)$ was called the
{\it subsymbol} of the operator $T$. We will need the explicit
formula for the subsymbol of a given second order differential
operator.

If $M$ is a contact manifold, then
every operator $T\in\D^2_{\l, \l}(M)$ can be written
(in many different ways) in the form:
\begin{equation}
\label{Rep}
T=L_{X_{\phi_1}} \circ L_{X_{\phi_2}}+L_{X_{\phi_3}}\circ L_{Y_1}+L_{Y_2}\circ L_{Y_3} +
L_{X_{\phi_4}}+L_{Y_4}+F,
\end{equation}
where each $Y_i$ is a vector field tangent to the contact distribution,
$X_\phi$ is the contact vector field with the contact Hamiltonian
$\phi\in{\cal F}_{-\frac{1}{\ell+1}}(M)$,
the Lie derivative $L$ is defined by Eq.~(\ref{LieEq}),
and~$F$ denotes the operator of multiplication by a function.

The explicit expression for the subsymbol of differential
operator~(\ref{Rep}) is as follows (see~\cite{CO12}):
\begin{equation}
\label{vfields}
{\mathrm s}\sigma(T)=\ts
\oh\bigl[X_{\phi_1},X_{\phi_2}\bigr]-
\bigl(\frac{\ell+1}{\ell+2}\bigr)\bigl(\l-\oh\bigr)X_{L_{Y_1}(\phi_3)}+
\oh\pi\bigl[Y_2,Y_3\bigr]+X_{\phi_4},
\end{equation}
where $L_{Y}(\phi)$ denotes the Lie derivative
of a $-\frac{1}{\ell+1}$-density $\phi$ along the vector field $Y$,
and $\pi:\mathrm{Vect}(M)\to{\cal K}(M)$ is defined in (\ref{PiPr}).

\begin{rem}
{\rm
Although it seems almost impossible, the map ${\mathrm s}\sigma$
defined by~(\ref{vfields}) is
well-defined.
In other words, it is independent of the choice of the vector fields
in the representation (\ref{Rep}) of the operator $T$.
This can be checked directly by rewriting it in local coordinates,
see formula~(\ref{ExPSS}) below.
Since the expression~(\ref{vfields}) is written using invariant terms,
it commutes with the action of contact diffeomorphisms.
Note also that the existence of such a map is indigenous to
contact geometry.
There is no similar map commuting with the full group of diffeomorphisms,
except for the principal symbol.
}
\end{rem}

\subsection{The Laplace-Beltrami operator on the space of weighted densities}

The classical Laplace-Beltrami operator acting on
the space of smooth functions is defined as follows
$$
\Delta_g(f)=d^*df.
$$
This operator is completely determined by the metric
${g}$.

We will go to a more general framework and
consider the generalized Laplace-Beltrami operator
acting on the space of weighted densities:
$$
\Delta_g^\l: {\cal F}_\l(M)\to{\cal F}_\l(M).
$$
The explicit formula of this operator
 is as follows:
\[
\Delta_g^\lambda(\phi\,\vol^\l)=
\left(
{g}^{ij}\nabla_i\nabla_j(\phi)+
\frac{n^2\lambda(\lambda-1)}{(n-1)(n+2)}R\phi
\right)\vol^\l,
\]
where $R$ is the scalar curvature
(see~\cite{DO}, Proposition 5.2).

\subsection{Calculating the subsymbol of the Laplace-Beltrami operator}
Recall that $M$ is a contact manifold
and $n=2\ell+1$.
It turns out that the contact Riemannian curl of a given metric $g$
is proportional to the subsymbol of the Laplace-Beltrami operator
associated with $g$.
This property can be considered as an equivalent definition
of the contact Riemannian curl.

\begin{thm}
One has
\begin{equation}
\label{ProPEq}
\textstyle
\mathrm{s}\sigma(\Delta^\lambda_g)=
\left(\frac{\ell+1}{\ell+2}\right)
\left(2\lambda-1\right)X_{A_{g,\Theta}}.
\end{equation}
\end{thm}

\begin{proof}
The proof is essentially a direct computation.

Let us choose local Darboux coordinates.
Every second order differential operator
can be written in these coordinates as:
$$
\begin{array}{rcl}
T&=&T_{2,0,0}\,\partial_z^2+T_{1,i,0}\,\partial_z\partial_{x_i}
+T_{1,0,i}\,\partial_z\partial_{y_i}+T_{0,ij,0}\,\partial_{x_i}\partial_{x_j}+
T_{0,i,j}\,\partial_{x_i}\partial_{y_j}+T_{0,0,ij}\,\partial_{y_i}\partial_{y_j}\\[4pt]
&&+T_{1,0,0}\,\partial_z+T_{0,i,0}\,\partial_{x_i}+T_{0,0,i}\,\partial_{y_i}+T_{0,0,0}.
\end{array}
$$
The coordinate formula of the subsymbol was calculated in \cite{CO12}:
\begin{equation}
\label{ExPSS}
\begin{array}{rcl}
\mathrm{s}\sigma(T) &=&
\frac{1+2\l(\ell+1)}{\ell+2}
\Bigl(
\partial_z(T_{2,0,0}-\oh{}y_iT_{1,i,0}+\oh{}x_iT_{1,0,i})\\[6pt]
&&\hskip1.6cm
+\partial_{x_i}(T_{1,i,0}-\oh{}y_jT_{0,ij,0}+\oh{}x_jT_{0,i,j})\\[6pt]
&&\hskip1.6cm
+\partial_{y_i}(T_{1,0,i}+\oh{}x_jT_{0,0,ij}-\oh{}y_jT_{0,j,i})\Bigr)\\[6pt]
&&
+T_{1,0,0}-\oh{}y_iT_{0,i,0}+\oh{}x_iT_{0,0,i}.
\end{array}
\end{equation}
One can check that this is exactly the same formula as (\ref{vfields}).

The expression of the generalized Laplace-Beltrami operator
$\Delta^\lambda$ in  local coordinates was calculated in~\cite{DO},
the result is:
\[
\Delta^\lambda_g=
{g}^{ij}\partial_i\partial_j-({g}^{jk}\Gamma^i_{jk}+
2\lambda {g}^{ij}\Gamma^k_{jk})\partial_i+\mathrm{(0-th\; order \; coefficients)}.
\]

Let us combine the above two formulas.
We obtain
$\mathrm{s}\sigma(\Delta^\lambda_g)=X_{\phi}$,
where $\phi$ is a weighted density of the form
\begin{equation}
\label{sgama}
\textstyle
\phi=
\left(\left(1-\frac{1+2\lambda(\ell+1)}{\ell+2}\right)
{g}^{jk}\Gamma^t_{jk}\theta_t+
\left(2\lambda-\frac{1+2\lambda(\ell+1)}{\ell+2}\right)
\Gamma^j_{ij}{g}^{it}\theta_t\right)
 \mathrm{vol}^{-\frac{1}{\ell+1}},
\end{equation}
and $X_\phi$ is the corresponding contact vector field.

Finally, taking into account the fact that ${g}^{ij}\partial_i(\theta_j)=0$,
for the Darboux form $\theta$,
 the expression (\ref{sgama}), after collecting the terms, coincides with
 $\left(\frac{\ell+1}{\ell+2}\right)
\left(2\lambda-1\right)A_{g,\Theta}$.
\end{proof}

\begin{cor}
For a generic metric,
$\mathrm{s}\sigma(\Delta^\l_g)=0$
if and only if $\l={\frac{1}{2}}$.
\end{cor}

\begin{rem}
{\rm
In differential geometry it is known that the space of half-densities and
the space of differential operators ${\cal D}_{\frac{1}{2},\frac{1}{2}}(M)$ acting on them
play a very special role.
In our context, the space of half-densities appears naturally.
}
\end{rem}

\section{Cotangent lift and the geodesic spray}  \label{SbundleSec}

In this section, we calculate the contact Riemannian curl on the unit sphere bundle
$STM$ over a Riemannian manifold $(M,{g})$.
The manifold $STM$ is a classical example of contact manifold,
and, furthermore, it is equipped with the canonical lift of the metric.
We prove that the contact Riemannian curl vanishes in this case.

Recall that the classical {\it geodesic spray} is the Hamiltonian vector field
on $TM$ with Hamiltonian $H(x,y)={g}_{ij}(x)\,y^iy^j,$
where $y^i$ are coordinates on the fibers;
the restriction of this vector field to $STM$ is an intrinsically defined
contact vector field.
It is not reasonable expect existence of another, independent, invariant contact vector field in this case. 

\subsection{Statement of the main result}
The Riemannian metric ${g}$ on $M$ has a canonical lift to~$STM$
that will be denoted by $\bar g$.
The main result of this section is as follows.

\begin{thm}
\label{BigThm}
The contact Riemannian curl on $(STM,\bar g)$ is identically zero.
\end{thm}

In order to prove this theorem, we will need explicit formulas
for the contact structure and the canonical Riemannian metric on $STM$.

\subsection{The coordinates on $STM$}
Let $(M,{g})$ be any Riemannian manifold of dimension $n$.
The Riemannian geometry of the sphere bundle $STM$
was studied in~\cite{Taha},
we will be using the notation of that work.

Denote by $(x^1,\ldots,x^n)$ a local coordinate system in $M$ and $(y^1,\ldots,y^n)$ the
Cartesian coordinates in the tangent space $T_xM$ at the point $x$ in $M$.
The coordinates $(x,y)$ are local coordinates on the tangent bundle on $TM$.
The unit sphere bundle $STM$ is a hypersurface of the tangent bundle
$T(M)$, singled out as the level surface of the Hamiltonian of the geodesic spray
$$
H(x,y)=1
$$
at every point.

\subsection{The contact structure of the sphere bundle $STM$}
The sphere bundle $STM$ is represented by parametric equations:
\[
x^h=x^h,\quad x^{\bar h}=y^h=y^h(x^i,u^{\kappa}),
\]
where $u^{\kappa}$ are local coordinates on the sphere.\footnote{
Following \cite{Taha}, we will adopt the following index gymnastics.
Capital Latin letters $A,B,\ldots$ run  $1$ to $2n$.
Small latin letters $i,j,\ldots$ run $1$ to $n$.
Barred Latin indices $\bar i, \bar j,\ldots$ run  $n+1$ to $2n$.
Some of the Greek letters $\alpha, \beta, \ldots$ run  $1$ to $2n-1$.
Some other Greek letters $\kappa, \lambda, \ldots$ run $n+1$ to $2n-1$. }
The tangent vectors $B^{A}_\alpha=\frac{\partial x^A}{\partial u^\alpha}$ of $STM$ in $T(M)$ are given by
\begin{equation}
\label{BComp}
\begin{array}{lcllcl}
B^{h}_i&=&\delta^h_i,& B^{h}_\lambda&=&0,\\[2mm]
B^{\bar h}_i&=&\partial_i y^h,& B^{\bar h}_\lambda&=&\partial_\lambda y^h.
\end{array}
\end{equation}
The square matrix $\left ( \begin{array}{c}
B^{A}_\alpha\\[2mm]
 C^A
\end{array} \right )$, where $C^i=0$ and $C^{\bar i}=y^i$, is
invertible at each point $x$ in $M$.
Its inverse will be the matrix $(B^\alpha_{A}, C_A)$, given by the equations:
\begin{equation}
\label{BInvComp}
\begin{array}{lcllcl}
B^{h}_{i}&=&\delta^h_i,& B^{h}_{\bar i}&=&0,\\[2mm]
B^{\kappa}_{i}&=&-B^{\bar h}_i B^\kappa_{\bar h},& B^{\kappa}_{\bar i},&&
\end{array}
\end{equation}
and $C_A=\left (
\begin{array}{c}
C_i\\[2mm]
C_{\bar i}
\end{array}
\right )$, where
$C_{\bar i}={g}_{ih}y^h$ and
$C_i=\Gamma^h_{rs} {g}_{hi}\,y^ry^s$.
The next formulas can be deduced from
Eqs. (\ref{BComp}), (\ref{BInvComp}), and are useful for what follows
\[
\begin{array}{lcllcllcl}
B^{\bar h}_{\lambda} B^{\kappa}_{\bar h}&=&
\delta^\kappa_\lambda,& y^h B^{\kappa}_{\bar h}&=&
0,&B^{\bar h}_{\lambda} B^{\lambda}_{\bar i}+y^h C_{\bar i}&=&
\delta^h_i,\\[2mm]
B^{\bar h}_{\lambda} C_{\bar h}&=&0,& y^h C_{\bar h},&=&1.&&&
\end{array}
\]
The Riemannian metric indentifies the tangent bundle $T(M)$
and the cotangent bundle $T^*(M),$
and hence induces a 1-form $\theta$ on $T(M)$, called
the {\it Liouville form}, which in local coordinates reads as follows:
\[
\theta={g}_{ij}y^j \, dx^i,
\]
Denote by $\bar \theta$ the restriction of the 1-form $\theta$ to the sphere bundle $STM$.
It is as follows:
\[
\bar \theta_\alpha=\theta_A B^{ A}_\alpha.
\]
Eq. (\ref{BComp}) imply that $\bar \theta_i={g}_{ij}y^j$ and  $\bar \theta_\kappa=0$.
\begin{lemma} The form $\bar \theta$ defines a contact structure on $STM$. The volume form associated with it reads (up to a factor) as:
\[
\Omega\; dx^1\wedge \cdots \wedge dx^n\wedge
du^{n+1}\wedge...\wedge du^{2n-1},
\]
where
$
\Omega=\mathrm{det }(B^{A}_\alpha,C^{A})\, \mathrm{det } ({g}_{ij}).
$

\end{lemma}
\begin{proof} 
This is well known, see~\cite{Taha},
and can also be checked by a direct computation.
\end{proof}

\subsection{The Riemannian metric on $STM$}

The Riemannian metric ${g}$ on $M$ can be extended to
a Riemannian metric $\bar{{g}}$ on the sphere bundle $STM$.
Explicitly, $\bar{{g}}$ is given by (cf. \cite{Taha}):
\[
\begin{array}{lcl}
\bar{{g}}_{ji}&=&{{g}}_{ji}+{{g}}_{ts}(\nabla_j y^t)(\nabla_i y^s),\\[2mm]
\bar {{g}}_{\mu i}&=&{{g}}_{ts}(\partial_\mu y^t)(\nabla_iy^s),\\[2mm]
\bar {{g}}_{\mu \lambda}&=&{{g}}_{ji}(\partial_\mu y^j)(\partial_\lambda y^i).
\end{array}
\]
The inverse of $\bar{{g}}$ is given by
\[
\begin{array}{lcl}
\bar {{g}}^{ji}&=&{{g}}^{ji},\\[2mm]
\bar {{g}}^{\lambda h}&=&-{{g}}^{hl}(\nabla_ly^i) B^\lambda_{\bar i},\\[2mm]
\bar {{g}}^{\lambda \kappa}&=&\left ({{g}}^{ih}+ {{g}}^{ts}(\nabla_t y^i)(\nabla_s y^h) \right ) B^\lambda_{\bar i}B^\kappa_{\bar h}.
\end{array}
\]
The Christoffel symbols associated with this metric are given by
\[
\begin{array}{ccl}
\bar \Gamma^h_{ji}&=&\Gamma^h_{ij}+\frac{1}{2}\left (
R_{r sj}^{ h}y^r\nabla_i y^s+R_{r si}^{h}y^r\nabla_j y^s \right ),\\[2mm]
\bar \Gamma^h_{\mu i}&=&\frac{1}{2}R_{r si}^{ h}y^r B_\mu^{\bar s},\\[2mm]
\bar \Gamma^h_{\mu \lambda}&=&0,\\[2mm]
\bar \Gamma^\kappa_{ji}&=&\left(
\nabla_j \nabla_i y^h+\frac{1}{2}R_{r ji}^{h}y^r- \frac{1}{2} R_{r ij}^{ h}y^r - \frac{1}{2} \left ( R_{r sj}^{ l}y^r\nabla_i y^s+ \frac{1}{2}  R_{r si}^{l }y^r\nabla_j y^s\right )\nabla_l y^h \right )B^\kappa_{\bar h},\\[2mm]
\bar \Gamma^\kappa_{\mu j}&=&\left (\partial_\mu \nabla_i y^h-\frac{1}{2} R_{r si}^{ l}y^r B_\mu^{\bar s} \nabla_l y^h \right )B^\kappa_{\bar h},\\[2mm]
\bar \Gamma^\kappa_{\mu \lambda}&=&(\partial_\mu \partial_\lambda y^h)B^\kappa_{\bar h}.
\end{array}
\]

\subsection{Proof of Theorem \ref{BigThm}}
We are ready to prove the main result of this section.

\begin{lemma}
\label{lemma1}
We have
\begin{equation}
\label{CalcEq}
\begin{array}{lcl}
y^i\partial_{i}(\Omega)&=&- y^i B_\lambda^{\bar h}\partial_i(B^\lambda_{\bar h})\,
\Omega+ 2y^i\Gamma_{i}\, \Omega-y^iy^hy^mg_{hm}\Gamma_{ir}^h\, \Omega,\\[2mm]
y^l (\nabla_l y^i)B^\lambda_{\bar i}\, \partial_{\lambda}(\Omega)&=&
-y^l (\nabla_l y^i) \partial_\lambda(B^\lambda_{\bar i})\, \Omega,\\[2mm]
y^l (\nabla_l y^k)(\partial_\mu \partial_\lambda y^j)\,B^\lambda_{\bar k} B^\mu_{\bar  j}
&=&-y^l (\nabla_l y^k)\partial_\lambda (B^\lambda_{\bar k}),\\[2mm]
\partial_\lambda(\nabla_l y^h) B^\lambda_{\bar h}y^l&=&
-(\partial_l B^\lambda_{\bar h}) B^{\bar h}_\lambda y^l+
\Gamma^i_{li} \,y^l-\Gamma^h_{rs}\,{g}_{ih}y^iy^ry^s.
\end{array}
\end{equation}
\end{lemma}
\begin{proof}
The first and the second lines of (\ref{CalcEq}) follow from the fact that
\begin{eqnarray*}
\label{F1} B^\lambda_{\bar k} \partial_\mu(B_\lambda^{\bar j})
&=&
-\partial_\mu(B^\lambda_{\bar k}) B_\lambda^{ j}-\partial_\mu(y^j y^m){g}_{mk},\\[4pt]
\label{F2} y^i\partial_i B_\lambda^{\;\; \bar j}
&=&
-y^iB^{\bar h}_{\lambda}B_\kappa^{\bar j}\partial_i(B^\kappa_{\bar h})-y^iy^jy^mg_{hm}\Gamma_{il}^h\, B_\lambda^{\bar l},
\end{eqnarray*}
and the property of the determinant.
The third line of (\ref{CalcEq}) follows from the fact that
$$
B_\lambda^{\bar h}B^\lambda_{\bar i}+y^h C_{\bar i}=\delta_i^h
$$
and applying to it the partial derivative $\partial_\mu$.
The fourth line of (\ref{CalcEq}) follows when we substitute the covariant derivative
$\nabla_l y^h=\partial_l y^h+\Gamma_{li}^h\,y^i$
and use  the third equation.
\end{proof}

By definition,
\[
A_{\bar {g},\Theta}=\bar {{g}}^{ih}\bar \nabla_i(\theta_h  \Omega^{-\frac{1}{n}})+\bar {{g}}^{\lambda h}\bar \nabla_\lambda (\theta_h  \Omega^{-\frac{1}{n}})+\bar {{g}}^{h \lambda}\bar \nabla_h (\theta_\lambda  \Omega^{-\frac{1}{n}})+\bar {{g}}^{\lambda \kappa}\bar \nabla_\lambda (\theta_\kappa  \Omega^{-\frac{1}{n}}).
\]
The last two summands vanish because $\theta_\kappa=0$.
Let us compute the first two summands seperately.
Applying the covariant derivative $\bar \nabla$, we get
\[
\begin{array}{lcl}
\bar {{g}}^{ih}\bar \nabla_i(\theta_h  \Omega^{-\frac{1}{n}})&=&\partial_i y^i \;  \Omega^{-\frac{1}{n}}+y^i\partial_i ( \Omega^{-\frac{1}{n}})-R_{q r s i} y^qy^r (\nabla_h y^s) {{g}}^{ih}\;  \Omega^{-\frac{1}{n}}\\[2mm]
&&+(1+\frac{1}{n})\Gamma_{ij}^jy^i\;  \Omega^{-\frac{1}{n}}+\frac{1}{n} \left (\frac{1}{2}
R_{r sh}^{ h} y^ry^l (\nabla_l y^s)+ \partial_\lambda(\nabla_l y^h)B^\lambda_{\bar h}y^l\right )\;  \Omega^{-\frac{1}{n}}.
\end{array}
\]
Similarly,
\[
\begin{array}{lcl}
\bar {{g}}^{\lambda h}\bar \nabla_\lambda (\theta_h  \Omega^{-\frac{1}{n}})
&=&
-\partial_i y^i  \;  \Omega^{-\frac{1}{n}}-\Gamma_{ij}^jy^i \;\Omega^{-\frac{1}{n}}-(\nabla_i y^k)
B^\lambda_{\;\; \bar k}y^i\partial_\lambda(\Omega^{-\frac{1}{n}})\\[2mm]
&&+\frac{1}{2}R_{rs k h}y^ry^s{{g}}^{hl}(\nabla_l y^k) \;
\Omega^{-\frac{1}{n}}- \frac{1}{2n} (\nabla_i y^k)y^i R_{r k h}^{ \bar h} y^r\;
\Omega^{-\frac{1}{n}}\\[2mm]
&&+ -\frac{1}{n} y^l (\nabla_l y^k)(\partial_\lambda \partial_\mu y^j)B^\lambda_{\bar k}B^\mu_{\bar j}\;\Omega^{-\frac{1}{n}}.
\end{array}
\]
By collecting the terms and using Lemma \ref{lemma1}, we finally obtain:
$$
A_{\bar{g},\Theta}=
-\frac{1}{2}\left(
R_{ilsj}\,y^iy^l(\nabla_h y^s)\,{{g}}^{jh}
\right)
\mathrm{vol}^{-\frac{1}{n}}\equiv0,
$$
since the curvature tensor $R_{ilsj}$ is antisymmetric in two first indices.

Theorem \ref{BigThm} is proved.

\section{Examples}  \label{Sphere}

We finish the paper with concrete examples of Riemannian curl
for the $3$-dimensional sphere (with two natural metrics)
and the $3$-dimensional ellipsoid with the standard metric.

\subsection{The sphere $S^3$}
Consider the sphere $S^3$ in the standard symplectic space $\Bbb R^4$.
It is endowed with the natural contact structure that
can be defined by the contact form
\[
\theta=dz+xdy-ydx,
\]
where $x,y$ and $z$ are affine coordinates on $S^3$.
More precisely, if $p_1,p_2,q^1,q^2$ are symplectic Darboux coordinates
on $\Bbb R^4$, then $x=\frac{p_1}{q^2},\,y=\frac{q^1}{q^2},\,z=-\frac{p_2}{q^2}$.

The restriction of the Euclidean metric to the
sphere $S^3$ takes the following form:
\[
\begin{array}{rcl}
{g}_{S^3}&=&
\displaystyle
F\left(\left(
y^2+z^2+1\right) dx^2+(x^2+z^2+1)dy^2+(x^2+y^2+1)dz^2\right . \\[10pt]
&&\displaystyle
\left .
-2\, x y\, dx dy-2\,xz \, dx dz-2\, yz \,dydz \right ),
\end{array}
\]
where $F= \left(\frac{1}{(x^2+y^2+z^2+1 )}\right)^{2}$.

Let us also consider another, conformally equivalent, metric on~$S^3$:
\[
\tilde {g}_{S^3}:=
\left (\frac{x^2+y^2+z^2+1}{\frac{x^2}{a}+\frac{y^2}{b}+\frac{z^2}{c}+1}\right ){g}_{S^3},
\]
which appeared in the context of integrable systems in \cite{Tab,MT},
see also~\cite{DV}.

\begin{prop} The following results hold.
\begin{enumerate}
\item[(i)]
In the case of the ``round'' metric ${g}_{S^3}$, we have
$A_{{g}_{S^3}}= 0$;
\item[(ii)]
In the case of the metric $\tilde {g}_{S^3}$, we have
$A_{\tilde {g}_{S^3}}= \frac{5}{2}
\left(
   \left(\frac{1}{b}-\frac{1}{a}\right) xy+\left(\frac{
   1}{c}-1\right) z\right)$.
\end{enumerate}
\end{prop}
\begin{proof}
Part (i) follows from Corollary~\ref{CoCu}.
Part (ii) can be obtained by a
straightforward computation using Eq. (\ref{sgama}).
\end{proof}

\subsection{The case of the ellipsoid $E^3(a,b,c)$}  \label{Ellipsoid}
Consider the $3$-dimensional ellipsoid endowed with the standard
metric
\[
\begin{array}{lll}
 {g}_{E^3_{a,b,c}}=&{g}_{x,x} dx^2+{g}_{y,y} dy^2+{g}_{z,z} dz^2 + {g}_{x,y} dx\, dy+{g}_{x,z}dx \, dz+ {g}_{y,z} dy\, dz,
\end{array}
\]
where

\[
\begin{array}{ccl}
{g}_{x,x}&=&\displaystyle \frac{\left(b^2 y^2+c^2 z^2+1\right)^2+a^4 x^2 \left(y^2+z^2+1\right) }{\left ((ax)^2+(by)^2+(cz)^2+1\right )^2},\\[4mm]
{g}_{y,y}&=&  \displaystyle \frac{ \left(a^2 x^2+c^2 z^2+1\right)^2+b^4 y^2 \left(x^2+z^2+1\right)}{\left ((ax)^2+(by)^2+(cz)^2+1\right )^{2}},\\[4mm]
{g}_{z,z}&=& \displaystyle \frac{\left(a^2 x^2+b^2 y^2+1\right)^2+c^4 z^2 \left(x^2+y^2+1\right)}{\left ((ax)^2+(by)^2+(cz)^2+1\right )^{2}}, \\[4mm]
{g}_{x,y}&=&\displaystyle -2 x y \frac{a^4 x^2-a^2 \left(z^2 (b^2-c^2)+b^2 -1\right)+b^2 \left(b^2
   y^2+c^2 z^2+1\right)}{ \left ((ax)^2+(by)^2+(cz)^2+1\right )^{2} },\\[4mm]
{g}_{x,z}&=& \displaystyle -2 x z \frac{a^4 x^2-a^2 \left(y^2 \left(c^2-b^2\right)+c^2-1\right)+c^2
   \left(b^2 y^2+c^2 z^2+1\right)}{\left ((ax)^2+(by)^2+(cz)^2+1\right )^{2}}, \\[4mm]
{g}_{y,z}&=& \displaystyle -2 y z \frac{b^4 y^2-b^2 \left(x^2 \left(c^2-a^2\right)+c^2-1\right)+c^2
   \left(a^2 x^2+c^2 z^2+1\right)}{\left ((ax)^2+(by)^2+(cz)^2+1\right )^{2}}.
\end{array}
\]
\begin{prop}
 We have
\[
\begin{array}{rcl}
A_{ {g}_{E^3_{a,b,c}}}&=&
a^4 (a^2 - b^2) (b^2 + 2 c^2 + 2) x^3 y +
 b^4 (a^2 - b^2) (a^2 + 2 c^2 + 2) x y^3  \\[6pt]
 &&
 + (a^2 - b^2) c^4 (2 + a^2 + b^2 + c^2) x y z^2 \\[6pt]
&&
- a^4 (c^2 - 1) (a^2 + 2 b^2 + c^2 + 1) x^2 z
 - b^4 (c^2 - 1) (2 a^2 + b^2 + c^2 + 1) y^2 z
\\[6pt]
&&
 - c^4 (c^2 - 1) (2 a^2 + 2 b^2 + 1) z^3
 + (a^2 - b^2) (a^2 + b^2 +
    2 c^2 + 1) x y \\[6pt]
&&- (c^2 - 1) (2 a^2 + 2 b^2 +
    c^2) z.
\end{array}
\]

\end{prop}
\begin{proof}
Straightforward computation using Eq. (\ref{sgama}).
\end{proof}

\medskip

\noindent \textbf{Acknowledgments}.
We  thank Dimitry Leites and Christian Duval their interest in this work
and careful reading of preliminary versions of it.
We are also grateful to
Charles Conley, Eugene Ferapontov and Serge Tabachnikov
for a number of fruitful discussions.
The first author was partially supported by the Grant
NYUAD 063.
The second author was partially supported by the Grant
PICS05974 ``PENTAFRIZ'', of CNRS.

\def\eightit{\it}
\def\bib{\bf}
\bibliographystyle{amsalpha}

\end{document}